\documentclass[12pt,reqno]{amsart}

\usepackage{amsthm,amssymb,amstext,amscd,euscript,mathrsfs,amsfonts,amsbsy,amsxtra,latexsym,amsmath}
\usepackage{fullpage}
\usepackage[english]{babel}
\usepackage[latin1]{inputenc}
\usepackage{verbatim}
\usepackage{graphicx} 
\usepackage[all]{xy} 
\usepackage{enumitem}
\usepackage{bbm}
\numberwithin{figure}{section} 
\allowdisplaybreaks

\numberwithin{equation}{section}
\newtheorem{theorem}{\textbf{Theorem}}
\numberwithin{theorem}{section}

\newtheorem{lemma}[theorem]{\textbf{Lemma}}

\theoremstyle{definition}
\newtheorem{definition}[theorem]{Definition}
\newtheorem{remark}{Remark}[section]

\newcommand{\bea}{\begin{eqnarray}} 
\newcommand{\eea}{\end{eqnarray}} 
\newcommand{\be}{\begin{equation}} 
\newcommand{\ee}{\end{equation}} 
\newcommand{\benn}{\begin{equation*}} 
\newcommand{\eenn}{\end{equation*}}

\title[short title]{Towards a simple characterization of the Chern-Schwartz-MacPherson class}
\author{James Fullwood and Dongxu Wang}
\address{Department of Mathematics\\ University of Hong Kong\\ Pokfulam Road, Hong Kong.}
\email{fullwood@maths.hku.hk}
\address{Department of Mathematics\\ Dongbei University of Finance and Economics\\ 217 Jianshan St, Shahekou District, Dalian, Liaoning, China.}
\email{dxwang1981@gmail.com}
\begin{document}

\maketitle

\begin{abstract}
For a large class of possibly singular complete intersections we prove a formula for their Chern-Schwartz-MacPherson classes in terms of a single blowup along a scheme supported on the singular loci of such varieties. In the hypersurface case our formula recovers a formula of Aluffi proven in 1996. As our formula is in no way tailored to the complete intersection hypothesis, we conjecture that it holds for all closed subschemes of a smooth variety. If in fact true, such a formula would provide a simple characterization of the Chern-Schwartz-MacPherson class which does not depend on a resolution of singularities. We also show that our formula may be suitably interpreted as the Chern-Fulton class of a scheme-like object which we refer to as an `$\mathfrak{f}$-scheme'. 
\end{abstract}

\section{Introduction}
Influenced by ideas of Grothendieck, in the 1960s Deligne conjectured the existence of a natural transformation
\[
c_*:\mathscr{C}\to H_*,
\]
where $\mathscr{C}$ is the covariant constructible function functor and $H_*$ is the integral homology functor, such that for a smooth complex algebraic variety $X$
\[
c_*(\mathbbm{1}_X)=c(TX)\cap [X] \in H_*X,
\] 
where $\mathbbm{1}_X$ denotes the indicator function of $X$ and $c(TX)\cap [X]$ denotes the total homological Chern class of $X$. Such a natural transformation is necessarily unique, and the existence of such would deem the class $c_*(\mathbbm{1}_X)$ for singular $X$ a natural extension of Chern class to the realm of singular varieties. Moreover, functoriality would imply
\[
\int_X c_*(\mathbbm{1}_X)=\chi(X),
\]    
where the integral sign is notation for taking the degree-zero component of a homology class and $\chi(X)$ denotes topological Euler characteristic with compact support, thus generalizing the Gau{\ss}-Bonnet-Chern theorem to the singular setting. In 1974 MacPherson explicitly constructed such a $c_*$, thus proving Deligne's conjecture \cite{RMCC}. Then in 1981, Brasselet and Schwartz showed that the class $c_*(\mathbbm{1}_X)$ was the Alexander-dual in relative cohomology of a class constructed by Schwartz in the 1960s using radial vector fields \cite{BSCC}, thus MacPherson's Chern class $c_*(\mathbbm{1}_X)$ eventually became known as the \emph{Chern-Schwartz-MacPherson class} -- or \emph{CSM class} for short -- which from here on will be denoted $c_{\text{SM}}(X)$. Kennedy then generalized MacPherson's construction to hold over an arbitrary algebraically closed field of characteristic zero, for which the integral homology functor is promoted to the Chow group functor $A_*$ \cite{KCC}, and this is the context in which we will work throughout. 

While the functorial CSM classes occupy a central role in the study of characteristic classes of singular varieties, MacPherson's construction of $c_*$ involves such machinery as the graph construction, local Euler obstructions and Nash blowups, thus rendering them quite difficult to define, let alone compute. As such, simple characterizations of the CSM class are quite desirable for anyone interested in their study. Aluffi has given two characterizations for a general variety over an algebraically closed field of characteristic zero in terms of a resolution of singularities \cite{ADF}\cite{ALCG}, and also a characterization for a subvariety of projective space in terms of its general linear sections \cite{AE}. Our aim here is then to suggest a characterization of the CSM class which does not depend on resolution of singularities, not only for the sake of simplicity, but in the hopes that such a characterization may be used to generalize the CSM class to fields of positive characteristic.

We take as our starting point a formula proven by Aluffi in the hypersurface case, which depends on the notion of the \emph{singular scheme} of a variety. For $X$ a subvariety of a smooth variety $M$ (over an algebraically closed field of characteristic zero), let $\mathcal{J}_X$ be the subscheme of $X$ whose ideal sheaf is locally generated by the $m\times m$ minors of the matrix valued function
\[
(a_{ij})=\frac{\partial F_i}{\partial x_j},
\]
where $F_i=0$ are the defining equations for $X$ and $m$ is the codimension of $X$ in $M$. We refer to $\mathcal{J}_X$ as the singular scheme of $X$, as it is an intrinsic subscheme of $X$ supported on its singular locus. In \cite{ASCC}, Aluffi proves that if $X$ is hypersurface then  
\begin{equation}\label{mf}
c_{\text{SM}}(X)=c(TM)\cap p_*\left(\frac{\widetilde{X}-E}{1+\widetilde{X}-E}\right),
\end{equation}
where $p:\widetilde{M}\to M$ is the blowup of $M$ along $\mathcal{J}_X$, $E$ denotes the class of the exceptional divisor of the blowup, $\widetilde{X}$ denotes the class of $p^{-1}(X)$, $(1+\widetilde{X}-E)^{-1}$ is notation for the inverse Chern class of $\mathscr{O}(\widetilde{X}-E)$ and $p_*$ denotes the proper pushforward of algebraic cycles associated with $p$ as defined in \S1.4 of \cite{IT}\footnote{The LHS of equation (\ref{mf}) actually denotes $\iota_*c_{\text{SM}}(X)$, where $\iota:X\hookrightarrow M$ denotes the natural inclusion.}. As none of the ingredients of the RHS of equation (\ref{mf}) depend on $X$ being a hypersurface, it is natural to surmise equation (\ref{mf}) still holds in higher codimension, but this is not so. Moreover, while many formulas for CSM classes have appeared in the literature in the hypersurface case \cite{PPCC}\cite{MM}\cite{ASCC}, little progress has been made in generalizing such formulas to higher codimension. In particular, in \cite{ASCC} Aluffi states of his formula ``We do not know whether our result is an essential feature of hypersurfaces, or whether a formula similar to (\ref{mf}) may compute (Chern-)Schwartz-MacPherson's class of arbitrary varieties. While this is a natural question, the approach of this paper does not seem well suited to address it....''. 

We conjecture that the key to generalizing formula (\ref{mf}) to higher codimension lies in the simple observation that in the hypersurface case, the blowup of $M$ along $\mathcal{J}_X$ coincides with the blowup of $M$ along the \emph{scheme-theoretic union} $X\cup \mathcal{J}_X$ (i.e., the scheme whose ideal sheaf is the product of the ideal sheaves of $X$ and $\mathcal{J}_X$), and that it is precisely the blowup along $X\cup \mathcal{J}_X$ one should use for the generalization of (\ref{mf}) to higher codimension. We provide evidence for this conjecture by proving it for complete intersections in arbitrary codimension which we refer to as \emph{almost smooth}. If $X$ is a complete intersection in $M$ of codimension $m$ we refer to it as almost smooth if there exist $m$ hypersurfaces $X_1,\ldots, X_m$ in $M$ such that $X=X_1\cap \cdots \cap X_m$ with $X_1\cap \cdots \cap X_{m-1}$ being smooth. All hypersurfaces are vacuously almost smooth. The main result of this note is then given by the following
\begin{theorem}\label{mt}
Let $X$ be an almost smooth complete intersection in a smooth variety $M$ and let $p:\widetilde{M}\to M$ be the blowup of $M$ along $X\cup \mathcal{J}_X$. Then 
\begin{equation}\label{mf1}
c_{\emph{SM}}(X)=c(TM)\cap p_*\left(\frac{\widetilde{X}-E}{1+\widetilde{X}-E}\right),
\end{equation}
where $\widetilde{X}$ and $E$ denote the classes of $p^{-1}(X)$ and the exceptional divisor of the blowup respectively.
\end{theorem}

The proof of Theorem~\ref{mt} is given in \S\ref{proof}. We note that the moniker almost smooth is not to imply that the singularities of an almost smooth complete intersection are necessarily mild. For example, for a hypersurface in projective space with arbitrary singularities all of its general linear sections are almost smooth complete intersections. The almost smooth assumption on $X$ implies that the ideal sheaf $\mathscr{I}=\mathscr{I}_X\cdot \mathscr{I}_{\mathcal{J}_X}$ of the closed subscheme $X\cup \mathcal{J}_X\hookrightarrow M$ is of \emph{linear type}, which means the canonical surjection
\begin{equation}\label{ilt}
\text{Sym}^d(\mathscr{I})\to \mathscr{I}^d 
\end{equation}
is an isomorphism for all $d$, and this is crucial for our proof of Theorem~\ref{mt}. An algorithm to compute CSM classes of almost smooth complete intersections was developed in \cite{MHP}, for which many explicit examples appeared. In any case, it would be surprising if Theorem~\ref{mt} didn't hold sans the almost smoothness assumption, as formula (\ref{mf1}) is in no way tailored to it.   

We conclude with $\S\ref{fs}$, where -- in the spirit of unification of Chern classes for singular varieties -- we show that formula $(\ref{mf1})$ may be interpreted as the \emph{Chern-Fulton class} of an object we refer to as an $\mathfrak{f}$-\emph{scheme}, which we think of as a scheme with `negatively thickened' components. 

\emph{A word of caution.} For the sake of aesthetics, we will often make an abuse of notation by making no notational distinction between a class and its pushforward and/or pull back by an inclusion map, and the same goes for bundles and their restrictions.

\section{Proof of main theorem}\label{proof}
Before proving Theorem~\ref{mt} we introduce some notation which will streamline our computations. So let $S$ be an algebraic scheme over a field, denote its Chow group by $A_*S$ and denote by $d$ the dimension of the largest component of $S$. For $\alpha=\sum_i\alpha^i\in A_*S$ with $\alpha^i\in A_{(d-i)}S$ we let 
\begin{equation}\label{dn}
\alpha^{\vee}=\sum_i(-1)^i\alpha^i, 
\end{equation}
and refer to it as the \emph{dual} of $\alpha$. We remark that by replacing $-1$ by a positive integer $n$ in formula (\ref{dn}) yields the $n$-th \emph{Adams} of $\alpha$, usually denoted $\alpha^{(n)}$. Thus we may think of $a^{\vee}$ as the `$-1$th Adams' of $\alpha$. For a line bundle $\mathscr{L}\to S$ we let 
\begin{equation}\label{tn}
\alpha \otimes \mathscr{L}=\sum_i\frac{\alpha^i}{c(\mathscr{L})^i},
\end{equation}
and refer to it as $\alpha$ \emph{tensor} $\mathscr{L}$. After identifying $\mathscr{L}$ with its class in the Picard group $\text{Pic}(S)$ it is then straightforward to show the map $\alpha\mapsto \alpha \otimes \mathscr{L}$ defines an action of $\text{Pic}(S)$ on $A_*S$, so that for any other line bundle $\mathscr{M}\to S$ we have
\begin{equation}\label{tf2}
(\alpha \otimes \mathscr{L})\otimes \mathscr{M}=\alpha \otimes (\mathscr{L}\otimes \mathscr{M}).
\end{equation}
This fact is proven in \cite{MFCH}, along with the fact that if $\mathscr{E}$ is a class in the Grothendieck group of vector bundles on $S$ then
\begin{equation}\label{df}
\left(c(\mathscr{E})\cap \alpha\right)^{\vee}=c(\mathscr{E}^{\vee})\cap \alpha^{\vee},
\end{equation}
and
\begin{equation}\label{tf}
\left(c(\mathscr{E})\cap \alpha\right) \otimes \mathscr{L}=\frac{c(\mathscr{E}\otimes \mathscr{L})}{c(\mathscr{L})^r}\cap \left(\alpha \otimes \mathscr{L}\right),
\end{equation} 
where $r\in \mathbb{Z}$ denotes the rank of $\mathscr{E}$.  

We will also need the following
\begin{definition}
Let $Y$ be a closed subscheme of $S$ .The \emph{Segre class} of $Y$ (relative to $S$) is denoted $s(Y,S)$, and is defined as
\begin{equation}\label{scd}
s(Y,S)=\begin{cases} c(N_YS)^{-1}\cap [Y] \quad \quad \quad \text{for $Y$ regularly embedded in $S$} \\ f_*\left(c(N_E\widetilde{S})^{-1}\cap [E]\right) \quad \hspace{1.05cm} \text{otherwise,}\end{cases}  
\end{equation}
where $c(N_YS)$ denotes the Chern class of the normal bundle to $Y$ in $S$ (in the case that $Y$ is regularly embedded), $f:\widetilde{S}\to S$ is the blowup of $S$ along $Y$ with exceptional divisor $E$ and $N_E\widetilde{S}$ denotes the normal bundle to $E$ in $\widetilde{S}$.
\end{definition}

We now recall the assumptions of Theorem~\ref{mt}. So let $M$ be a smooth variety over an algebraically closed field of characteristic zero, and let $X=X_1\cap \cdots \cap X_m$ be a complete intersection of $m$ hypersurfaces in $M$ such that $Z=X_1\cap \cdots \cap X_{m-1}$ is smooth. We denote the blowup of $M$ along $X\cup \mathcal{J}_X$ by $p:\widetilde{M}\to M$, where $\mathcal{J}_X$ denotes the singular scheme of $X$ and $X\cup \mathcal{J}_X$ is the subscheme of $M$ corresponding to the ideal sheaf $\mathscr{I}_X\cdot \mathscr{I}_{\mathcal{J}_X}$, i.e., the product of the ideal sheaves of $X$ and $\mathcal{J}_X$. We first prove the following
\begin{lemma}\label{l1}
Let $k$ be a positive integer and let $X\cup \mathcal{J}_X^k$ be the closed subscheme of $M$ corresponding to the ideal sheaf $\mathscr{I}_X\cdot \mathscr{I}^k_{\mathcal{J}_X}$. Then
\[
s(X\cup \mathcal{J}_X^k,Z)=s(X,Z)+c(\mathscr{O}(X))^{-1}\cap \left(s(\mathcal{J}_X,Z)^{(k)}\otimes_Z \mathscr{O}(X)\right),
\] 
where $\mathscr{O}(X)$ is the line bundle on $Z$ corresponding to the divisor $X$, and $s(\mathcal{J}_X,Z)^{(k)}$ denotes the $k$th-Adams of $s(\mathcal{J}_X,Z)$.
\end{lemma}
\begin{proof}
The proof of this fact follows along the lines of the proof of Proposition~3 in \cite{MFCH}. Indeed, let $q:\widetilde{Z}\to Z$ be the blowup of $Z=X_1\cap \cdots \cap X_{m-1}$ along $\mathcal{J}_X$, viewing $\mathcal{J}_X$ as a subscheme of $Z$, denote the exceptional divisor of the blowup by $E$ and denote the class of $q^{-1}(X)$ by $\widetilde{X}$. 

By Proposition~4.2 (a) of \cite{IT}, we have
\begin{eqnarray*}
s(X\cup \mathcal{J}^k_X,Z)&=&q_*\left(s(q^{-1}(X\cup \mathcal{J}^k_X),\widetilde{Z})\right) \\
                          &=&q_*\left(\frac{\widetilde{X}+kE}{1+\widetilde{X}+kE}\right) \\
                          &=&q_*\left((\widetilde{X}+kE)\otimes \mathscr{O}(\widetilde{X}+kE)\right),
\end{eqnarray*}
where the last equality comes from formula (\ref{tn}). Now we compute:
\begin{eqnarray*}
q_*\left((\widetilde{X}+kE)\otimes \mathscr{O}(\widetilde{X}+kE)\right)&=&q_*\left((\widetilde{X}+kE)\otimes (\mathscr{O}(kE)\otimes \mathscr{O}(\widetilde{X}))\right) \\
                                                 &=&q_*\left(\left(\widetilde{X}\otimes \mathscr{O}(kE)+kE\otimes \mathscr{O}(kE)\right)\otimes \mathscr{O}(\widetilde{X})\right) \\
																								 &=&q_*\left(\left(\frac{\widetilde{X}}{1+kE}+\frac{kE}{1+kE}\right)\otimes \mathscr{O}(\widetilde{X})\right) \\
																								 &=&q_*\left(\left(\widetilde{X}-\frac{\widetilde{X}\cdot kE}{1+kE}+\frac{kE}{1+kE}\right)\otimes \mathscr{O}(\widetilde{X})\right) \\
																								 &=&q_*\left(\left(\widetilde{X}+c\left(\mathscr{O}(-\widetilde{X})\right)\cap \frac{kE}{1+kE}\right)\otimes \mathscr{O}(\widetilde{X})\right) \\
																								 &=&q_*\left(s(\widetilde{X},\widetilde{Z})+c(\mathscr{O}(\widetilde{X}))^{-1}\cap \left(s(E,\widetilde{Z})^{(k)}\otimes \mathscr{O}(\widetilde{X})\right)\right) \\
																								 &=&s(X,Z)+c(\mathscr{O}(X))^{-1}\cap \left(s(\mathcal{J}_X,Z)^{(k)}\otimes \mathscr{O}(X)\right),
\end{eqnarray*}
where $s(\mathcal{J}_X,Z)^{(k)}$ denotes the $k$th Adams of $s(\mathcal{J}_X,Z)$. The second to last equality follows from the definition of Segre class and formula \ref{tf}, while the last equality follows from the projection formula. This proves the lemma.
\end{proof}

We now proceed with the
\begin{proof}[Proof of Theorem~\ref{mt}]
For every positive integer $k$ we may factor the embedding $X\cup \mathcal{J}^k_X\hookrightarrow M$ as $X\cup \mathcal{J}^k_X\hookrightarrow Z\hookrightarrow M$, and since $Z$ is smooth, this implies the ideal sheaf of $X\cup \mathcal{J}^k_X$ is of linear type (as defined via equation (\ref{ilt})), thus by Theorem~2 of \cite{KLI} we have
\[
s(X\cup \mathcal{J}^k_X,M)=c(N_ZM)^{-1}\cap s(X\cup \mathcal{J}^k_X,Z).
\] 
Moreover, by the birational invariance of Segre classes (Proposition~4.2 (a) of \cite{IT}), we have (recall $p$ denotes the blowup of $M$ along $X\cup \mathcal{J}_X$)
\[
s(X\cup \mathcal{J}^k_X,M)=p_*\left(\frac{\widetilde{X}+kE}{1+\widetilde{X}+kE}\right),
\]
where $E$ and $\widetilde{X}$ denote the class of the exceptional divisor and the pullback of $X$ respectively. Thus by Lemma~\ref{l1} we have 
\begin{equation}\label{e1}
p_*\left(\frac{\widetilde{X}+kE}{1+\widetilde{X}+kE}\right)=c(N_ZM)^{-1}\cap \left(s(X,Z)+c(\mathscr{O}(X))^{-1}\cap \left(s(\mathcal{J}_X,Z)^{(k)}\otimes_Z \mathscr{O}(X)\right)\right),
\end{equation}
where we use a subscript $Z$ on the tensor notation to emphasize the fact that we are tensoring the class $s(\mathcal{J}_X,Z)^{(k)}$ with respect to codimension in $Z$. Now for $k=-1$, the geometric meaning of the equation (\ref{e1}) as the Segre class of a closed subscheme of $M$ is lost, but equation (\ref{e1}) still holds nonetheless. And since the dual of a class may be suitably interpreted as its `$-1$th Adams', we have
\begin{equation}\label{e2}
p_*\left(\frac{\widetilde{X}-E}{1+\widetilde{X}-E}\right)=c(N_ZM)^{-1}\cap \left(s(X,Z)+c(\mathscr{O}(X))^{-1}\cap \left(s(\mathcal{J}_X,Z)^{\vee}\otimes_Z \mathscr{O}(X)\right)\right),
\end{equation}
where we recall $s(\mathcal{J}_X,Z)^{\vee}$ denotes the dual of $s(\mathcal{J}_X,Z)$ (\ref{dn}). Now let 
\[
\mathscr{E}=\mathscr{O}(X_1) \oplus \cdots \oplus \mathscr{O}(X_m),
\]
and note that the restriction to $Z$ of the bundles $\mathscr{O}(X_1)\oplus \cdots \oplus \mathscr{O}(X_{m-1})$ and $\mathscr{O}(X_m)$ coincide with its normal bundle $N_ZM$ and $\mathscr{O}(X)$ respectively, so that $c(N_ZM)c(\mathscr{O}(X))=c(\mathscr{E})$.
Thus
\begin{eqnarray*}
p_*\left(\frac{\widetilde{X}-E}{1+\widetilde{X}-E}\right)&=&c(N_ZM)^{-1}\cap \left(s(X,Z)+c(\mathscr{O}(X))^{-1}\cap \left(s(\mathcal{J}_X,Z)^{\vee}\otimes_Z \mathscr{O}(X)\right)\right) \\
&=&s(X,M)+c(\mathscr{E})^{-1}\cap \left(s(\mathcal{J}_X,Z)^{\vee}\otimes_Z \mathscr{O}(X)\right) \\
&=&s(X,M)+c(\mathscr{E})^{-1}\cap \left(\left(c(N_ZM)\cap s(\mathcal{J}_X,M)\right)^{\vee}\otimes_Z \mathscr{O}(X)\right) \\
&=&s(X,M)+c(\mathscr{E})^{-1}\cap \left(\left(c(N_ZM)^{\vee}\cap s(\mathcal{J}_X,M)^{\vee}\right)\otimes_Z \mathscr{O}(X)\right), \\
\end{eqnarray*}
where in the second and third equalities we used the fact that
\[
c(N_ZM)^{-1}\cap s(X,Z)=s(X,M) \quad \text{and} \quad s(\mathcal{J}_X,Z)=c(N_ZM)\cap s(\mathcal{J}_X,M)
\]
(both of which follow by Theorem~2 of \cite{KLI}), and in the fourth equality we use formula (\ref{df}). Our theorem is then proved once we show
\[
c(TM)\cap \left(s(X,M)+c(\mathscr{E})^{-1}\cap \left(\left(c(N_ZM)^{\vee}\cap s(\mathcal{J}_X,M)^{\vee}\right)\otimes_Z \mathscr{O}(X)\right)\right)=c_{\text{SM}}(X).
\]
For this, let 
\[
M(X)=c(\mathscr{E})^{-1}\cap \left(\left(c(N_ZM)^{\vee}\cap s(\mathcal{J}_X,M)^{\vee}\right)\otimes_Z \mathscr{O}(X)\right).
\]
As Theorem~1.1 of \cite{FMC} is equivalent to the statement that
\[
c_{\text{SM}}(X)-c(TM)\cap s(X,M)=c(TM)\cap \left((-1)^{m-1}\frac{c(\mathscr{E}^{\vee}\otimes \mathscr{O}(X_m))}{c(\mathscr{E})}\cap \left(s(\mathcal{J}_X,M)^{\vee}\otimes \mathscr{O}(X_m)\right)\right),
\]
showing
\[
M(X)=(-1)^{m-1}\frac{c(\mathscr{E}^{\vee}\otimes \mathscr{O}(X_m))}{c(\mathscr{E})}\cap \left(s(\mathcal{J}_X,M)^{\vee}\otimes \mathscr{O}(X_m)\right)
\] 
then finishes the proof of the theorem. Indeed,
\begin{eqnarray*}									
M(X)&=&c(\mathscr{E})^{-1}\cap \left(\left(c(N_ZM^{\vee})\cap s(\mathcal{J}_X,M)^{\vee}\right)\otimes_Z \mathscr{O}(X)\right) \\
&=&(-1)^{m-1}\frac{c(\mathscr{O}(X_m))^{m-1}}{c(\mathscr{E})}\cap \left((c(N_ZM^{\vee})\cap s(\mathcal{J}_X,M)^{\vee})\otimes_M \mathscr{O}(X_m)\right) \\
&\overset{(\ref{tf})}=&(-1)^{m-1}\frac{c(\mathscr{O}(X_m))^{m-1}}{c(\mathscr{E})}\cap \left(\frac{c(N_ZM^{\vee}\otimes \mathscr{O}(X_m))}{c(\mathscr{O}(X_m))^{m-1}}\cap (s(\mathcal{J}_X,M)^{\vee}\otimes_M \mathscr{O}(X_m))\right)\\
&=&(-1)^{m-1}\frac{c(\mathscr{E}^{\vee}\otimes \mathscr{O}(X_m))}{c(\mathscr{E})}\cap(s(\mathcal{J}_X,M)^{\vee}\otimes_M \mathscr{O}(X_m)),
\end{eqnarray*}
where in the second equality we used the fact that for all classes $\alpha$ we have (see Lemma~2.1 of \cite{FMC})
\[
\alpha\otimes_{Z} \mathscr{O}(X) =c(\mathscr{O}(X_m))^{m-1} \cap \left(\alpha\otimes_M \mathscr{O}(X_m)\right),
\]
and the factor of $(-1)^{m-1}$ appears due to the fact that we switch from taking duals in $Z$ to duals in $M$. In the last equality we used that since $\mathscr{E}=N_ZM\oplus \mathscr{O}(X_m)$, $\mathscr{E}^{\vee}\otimes \mathscr{O}(X_m)=(N_ZM^{\vee}\otimes \mathscr{O}(X_m))\oplus \mathscr{O}$, thus
\[
c(\mathscr{E}^{\vee}\otimes \mathscr{O}(X_m))=c(N_ZM^{\vee}\otimes \mathscr{O}(X_m)),
\]
which concludes the proof.
\end{proof}

\section{CSM classes via Chern-Fulton classes of $\mathfrak{f}$-schemes}\label{fs}
There are various notions of Chern class for singular varieties and schemes which coincide with the usual Chern class in the smooth case. One such class is the \emph{Chern-Fulton class}, which is defined for all closed subschemes of a smooth variety $M$ over an arbitrary field. In contrast to the CSM classes, they are easy to define. In particular, for a closed subscheme $V\hookrightarrow M$ its Chern-Fulton class $c_{\text{F}}(V)$ is given by 
\[
c_{\text{F}}(V)=c(TM)\cap s(V,M).
\] 
In Example~4.2.6 of \cite{IT}, Fulton proves his classes are intrinsic to the scheme $V$, and are thus independent of any embedding of $V$ in a smooth variety. While CSM classes generalize the Gau{\ss}-Bonnet-Chern theorem to singular varieties, Chern-Fulton classes provide a deformation-invariant extension of the Gau{\ss}-Bonnet-Chern theorem, since if $\mathcal{Z}\to \Delta$ is a family over a disk $\Delta\subset \mathbb{C}$ whose fibers are all smooth except for possibly the central fiber, then for all $t\neq 0$ in $\Delta$ we have
\[
\int_{Z_0} c_{\text{F}}(Z_0)=\chi(Z_t),
\]
where $Z_0$ denotes the central fiber of the family and $Z_t$ denotes the fiber over $t\neq 0$. Furthermore, Chern-Fulton classes are sensitive to possible non-reduced scheme structure, while it is known that the CSM class of a non-reduced scheme coincides with the CSM class of its reduced support. In spite of these differences however, the Chern-Schwartz-MacPherson class and the Chern-Fulton class are closely related. In particular, if $X$ is a complex hypersurface with isolated singularities then 
\[
\int_X c_{\text{SM}}(X)-c_{\text{F}}(X)=\sum_{x_i\in \text{Sing}(X)} \mu(x_i), 
\]
where $\mu(x_i)$ denotes the \emph{Milnor number} of the singular point $x_i\in \text{Sing}(X)$. For arbitrary $X$ the class
\[
\mathcal{M}(X)=c_{\text{SM}}(X)-c_{\text{F}}(X)
\] 
is then an invariant of the singularities of $X$ which generalizes the notion of global Milnor number to all varieties and schemes, and is referred to as the \emph{Milnor class} of $X$. In this section we take this relationship between the two classes a step further, and show how the main formula (\ref{mf1}) of Theorem~\ref{mt} may be interpreted as the Chern-Fulton class of a formal object we refer to as an $\mathfrak{f}$-\emph{scheme}. 

So let $X$ be an almost smooth complete intersection in a smooth variety $M$ (over an algebraically closed field of characteristic zero), and let $p:\widetilde{M}\to M$ denote the blowup of $M$ along the singular scheme $\mathcal{J}_X$ of $X$. We recall that for all $k>0$ Proposition~4.2 (a) of \cite{IT} implies
\[
p_*\left(\frac{\widetilde{X}+kE}{1+\widetilde{X}+kE}\right)=s(X\cup \mathcal{J}^k_X,M),
\]
where $X\cup \mathcal{J}^k_X$ denotes the closed subscheme of $M$ corresponding to the ideal sheaf $\mathscr{I}_X\cdot \mathscr{I}^k_{\mathcal{J}_X}$. It immediately follows that
\[
c_{\text{F}}(X\cup \mathcal{J}^k_X)=c(TM)\cap p_*\left(\frac{\widetilde{X}+kE}{1+\widetilde{X}+kE}\right).
\]
Now view $k$ as a parameter. For $k=0$ we have 
\[
c_{\text{F}}(X)=\left.c(TM)\cap p_*\left(\frac{\widetilde{X}+kE}{1+\widetilde{X}+kE}\right)\right|_{k=0},
\]
while Theorem~\ref{mt} amounts to the assertion
\begin{equation}\label{ae1}
c_{\text{SM}}(X)=\left.c(TM)\cap p_*\left(\frac{\widetilde{X}+kE}{1+\widetilde{X}+kE}\right)\right|_{k=-1},
\end{equation}
so that $c_{\text{SM}}(X)$, $c_{\text{F}}(X)$ and $c_{\text{F}}(X\cup \mathcal{J}^k_X)$ all correspond to evaluating a single expression at different values of $k$. Moreover, since we think of the scheme $X\cup \mathcal{J}^k_X$ as a $k$th thickening of $X$ along $\mathcal{J}_X$, the CSM class of $X$ may be interpreted formally as the Chern-Fulton class of a scheme-like object which is a `negative thickening' of $X$ along $\mathcal{J}_X$, or rather, as a geometric object associated with the `fraction' $\mathscr{I}_X\cdot \mathscr{I}^{-1}_{\mathcal{J}_X}$. We make this qualitative interpretation more precise as follows.

We first need the following
\begin{definition}
Let $\mathfrak{I}$ denote the monoid generated by quasicoherent ideal sheaves over $M$, with binary operation induced by the usual product of ideal sheaves. The group of \emph{$\mathfrak{f}$-schemes} of $M$, denoted $\mathfrak{F}(M)$, is then defined as the Grothendieck group of the monoid $\mathfrak{I}$. 
\end{definition}

\noindent \emph{Some terminology}. An $\mathfrak{f}$-scheme of the form $[\mathscr{I}_S]^{-1}$ with $S$ non-empty will be referred to as an \emph{antischeme}. An $\mathfrak{f}$-scheme $[\mathscr{I}_A]$ different from the identity is said to be a \emph{factor} of $[\mathscr{I}_S]$ if there exists a non-empty closed subscheme $B\hookrightarrow M$ such that $[\mathscr{I}_S]=[\mathscr{I}_A]\cdot [\mathscr{I}_B]$. An $\mathfrak{f}$-scheme $U\in \mathfrak{F}(M)$ is said to have \emph{support} $S\cup T$ if there exists closed subschemes $S\hookrightarrow M$ and $T\hookrightarrow M$ such that $U=[\mathscr{I}_S]\cdot [\mathscr{I}_{T}]^{-1}$ with $[\mathscr{I}_S]$ and $[\mathscr{I}_T]$ having no common factors. In such a case we will often make an abuse of notation and denote $[\mathscr{I}_S]\cdot [\mathscr{I}_{T}]^{-1}\in \mathfrak{F}(M)$ with support $S\cup T$ simply by $S\cdot T^{-1}$.
\\

We now extend the domain of Segre classes and Chern-Fulton classes to $\mathfrak{f}$-schemes via the following
\begin{definition}\label{d2}
Let $U=[\mathscr{I}_{S_1}]\cdot [\mathscr{I}_{S_2}]^{-1}\in \mathfrak{F}(M)$ be an $\mathfrak{f}$-scheme with support $S_1\cup S_2$, and let $p:\widetilde{M}\to M$ be the blowup of $M$ along $S_1\cup S_2$. Then the \emph{Segre class} of $U$ is defined via the formula
\[
s(U,M)=p_*\left(\frac{\widetilde{S_1}-\widetilde{S_2}}{1+\widetilde{S_1}-\widetilde{S_2}}\right),
\]
where $\widetilde{S_i}$ denotes the class of $p^{-1}(S_i)$. The Chern-Fulton class of $U$ is then given by
\[
c_{\text{F}}(U)=c(TM)\cap s(U,M).
\]
\end{definition}

We note that if $S_2$ in Definition~\ref{d2} is the empty subscheme of $M$ then $s(U,M)$ coincides with the usual Segre class $s(S_1,M)$, and if $S_1$ is empty then $s(U,M)=s(S_2,M)^{\vee}$. 

Now let $X$ be an almost smooth complete intersection in $M$ with singular scheme $\mathcal{J}_X$. Then the RHS of equation (\ref{ae1}) coincides with the Chern-Fulton class of the $\mathfrak{f}$-scheme $X\cdot \mathcal{J}_X^{-1}$, thus Theorem~\ref{mt} may be reformulated in the language of $\mathfrak{f}$-schemes via the formula
\begin{equation}\label{mf2}
c_{\text{SM}}(X)=c_{\text{F}}(X\cdot \mathcal{J}_X^{-1}),
\end{equation}
where we recall $X\cdot \mathcal{J}_X^{-1}$ is notation for the $\mathfrak{f}$-scheme $[\mathscr{I}_X]\cdot [\mathscr{I}_{\mathcal{J}_X}]^{-1}$. As for the Milnor class, we then have
\[
\mathcal{M}(X)=c_{\text{F}}(X\cdot \mathcal{J}_X^{-1})-c_{\text{F}}(X),
\]
so that $c_{\text{SM}}(X)$, $c_{\text{F}}(X)$ and $\mathcal{M}(X)$ may all be formulated in terms of Chern-Fulton classes of $\mathfrak{f}$-schemes. Moreover, we conjecture formula (\ref{mf2}) holds for $X$ \emph{any} closed subscheme of $M$.

\begin{remark} The language of $\mathfrak{f}$-schemes along with formula (\ref{mf2}) yields a simple proof in the hypersurface case that CSM classes are not sensitive to non-reduced scheme structure. Indeed, let $X$ be a reduced hypersurface given by the equation $F=0$ and denote its singular scheme by $\mathcal{J}_X$. Then the scheme $X^k$ corresponding to the ideal sheaf $\mathscr{I}_X^k$ is given by $F^k=0$, and $d(F^k)=kF^{k-1}dF$. Now since $\mathcal{J}_X$ corresponds to the equation $dF=0$, the singular scheme of $X^k$ corresponds to the ideal sheaf $\mathscr{I}_X^{k-1}\cdot \mathscr{I}_{\mathcal{J}_X}$. Thus
\[
c_{\text{SM}}(X^k)=c_{\text{F}}(X^k\cdot (X^{(k-1)}\cdot \mathcal{J}_X)^{-1})=c_{\text{F}}(X^k\cdot (X^{(1-k)}\cdot \mathcal{J}_X^{-1}))=c_{\text{F}}(X\cdot \mathcal{J}_X^{-1})=c_{\text{SM}}(X),
\] 
where the first and last equalities follow from formula (\ref{mf2}).
\end{remark}

We conclude with a quote from 18th century mathematician Fancis Maceres in regards to negative numbers:
\\

``Quantities marked with a minus sign darken the very whole doctrines of the equations, and make dark of the things which are in their nature excessively obvious and simple.''

\bibliographystyle{plain}
\bibliography{CSMC2}

\end{document}